\documentclass[11pt]{amsart}

\usepackage{mathrsfs}
\usepackage{amsfonts}
\usepackage{amssymb}
\usepackage{amsxtra}
\usepackage{dsfont}
\usepackage[dvipsnames]{xcolor}
\usepackage[english,polish]{babel}
\usepackage{enumerate}
\usepackage[shortlabels]{enumitem}
\usepackage{todonotes}

\usepackage{tikz}
\usetikzlibrary{arrows}

\usepackage[compress, sort]{cite}

\usepackage[T1]{fontenc}

\usepackage{newtxmath}
\usepackage[type1]{newtxtext}

\usepackage[margin=2.5cm, centering]{geometry}
\usepackage[colorlinks,citecolor=blue,urlcolor=blue,bookmarks=true]{hyperref}
\hypersetup{
pdfpagemode=UseNone,
pdfstartview=FitH,
pdfdisplaydoctitle=true,
pdfborder={0 0 0}, % No link borders
pdftitle={Nevai's condition for measures with unbounded supports},
pdfauthor={Grzegorz Świderski},
pdflang=en-US
}

\newcommand{\CC}{\mathbb{C}}
\newcommand{\ZZ}{\mathbb{Z}}

\newcommand{\NN}{\mathbb{N}}
\newcommand{\RR}{\mathbb{R}}
\newcommand{\PP}{\mathbb{P}}
\newcommand{\EE}{\mathbb{E}}

\newcommand{\calC}{\mathcal{C}}
\newcommand{\calL}{\mathcal{L}}

\newcommand{\calF}{\mathcal{F}}

\newcommand{\calX}{\mathcal{X}}

\newcommand{\calD}{\mathcal{D}}

\newcommand{\frakB}{\mathfrak{B}}
\newcommand{\frakX}{\mathfrak{X}}

\newcommand{\Var}{\operatorname{Var}}

\newcommand{\Id}{\operatorname{Id}}

\newcommand{\sign}[1]{\operatorname{sign}({#1})}

\newcommand{\norm}[1]{\lVert {#1} \rVert}

\newcommand{\tr}{\operatorname{tr}}

\DeclareMathOperator{\sinc}{sinc}
\DeclareMathOperator{\intr}{int}

\newcommand{\cl}[1]{\operatorname{cl}({#1})}

\newcommand{\discr}{\operatorname{discr}}

\newcommand{\muAC}{{\mu_{\mathrm{ac}}}}
\newcommand{\mupAC}{{\mu'_{\mathrm{ac}}}}
\newcommand{\muS}{{\mu_{\mathrm{s}}}}

\newcommand{\ud}{{\: \rm d}}
\newcommand{\ue}{\textrm{e}}
\newcommand{\supp}{\operatornamewithlimits{supp}}

\newtheorem{theorem}{Theorem}[section]

\newtheorem{proposition}[theorem]{Proposition}
\newtheorem{lemma}[theorem]{Lemma}
\newtheorem{corollary}[theorem]{Corollary}
\newtheorem{claim}[theorem]{Claim}

\theoremstyle{plain}
\newcounter{thm}

\numberwithin{equation}{section}

\theoremstyle{definition}
\newtheorem{example}[theorem]{Example}

\newtheorem{definition}[theorem]{Definition}

\title{Nevai's condition for measures with unbounded supports}

\author{Grzegorz \'{S}widerski}
\address{
    Grzegorz \'{S}widerski \\
    Institute of Mathematics \\ 
    Polish Academy of Sciences \\
    ul. Śniadeckich 8 \\
    00-696 Warsaw, Poland
}
\curraddr{Faculty of Pure and Applied Mathematics, Wroclaw University of Science and Technology, Wyb. Wyspiańskiego 27, 50-370 Wroclaw, Poland}
\email{grzegorz.swiderski@pwr.edu.pl}


\begin{document}
\selectlanguage{english}

\begin{abstract}
We study Nevai's condition from the theory of orthogonal polynomials on the real line. We prove that a large class of measures with unbounded Jacobi parameters satisfies Nevai's condition locally uniformly on the support of the measure away from a finite explicit set. This allows us to give applications to relative uniform and weak asymptotics of Christoffel--Darboux kernels on the diagonal and to limit theorems for unconventionally normalized global linear statistics of orthogonal polynomial ensembles.
\end{abstract}

\maketitle

\section{Introduction} \label{sec:1}
Let $\mu$ be a Borel probability measure on the real line with infinite support such that all the moments
\[
	\int_\RR x^n \ud \mu(x), \quad n \geq 0
\]
exist and are finite. Let $L^2(\mu)$ be the usual Hilbert space associated with the scalar product
\[
	\langle f, g \rangle_{L^2(\mu)} =
	\int_\RR f(x) \overline{g(x)} \ud \mu(x).
\]
By performing Gram--Schmidt orthonormalization process on the sequence $(x^n : n \geq 0)$ of monomials we get the sequence of \emph{orthonormal polynomials} $(p_n : n \geq 0)$. By the construction, for any $n \geq 0$ the polynomial $p_n$ is of exact degree $n$, its leading coefficient is positive, and for any $n \geq 0, m \geq 0$ we have
\begin{equation} \label{eq:1}
	\langle p_n, p_m \rangle_{L^2(\mu)} =
	\begin{cases}
		1 & n=m, \\
		0 & \text{otherwise}.
	\end{cases}
\end{equation}
Since the polynomial $xp_n$ is of degree $n+1$, by orthogonality we easily get the following \emph{three-term recurrence relation}
\begin{gather}
	\label{eq:2}
		p_0(x) \equiv 1, \quad p_1(x) = \frac{x-b_0}{a_0}, \\
	\label{eq:3}
	x p_n(x) = a_{n-1} p_{n-1}(x) + b_n p_n(x) + a_n p_{n+1}(x), \quad n \geq 1,
\end{gather}
where
\[
	a_n = \langle x p_n, p_{n+1} \rangle_{L^2(\mu)} > 0, \quad
	b_n = \langle x p_n, p_{n} \rangle_{L^2(\mu)} \in \RR, \quad n \geq 0
\]
are the so-called \emph{Jacobi parameters} of $\mu$.

One can reverse this procedure and for any given sequences $(a_n : n \geq 0), (b_n : n \geq 0)$ of positive and real numbers, respectively, define the sequence of polynomials $(p_n : n \geq 0)$ by formulas \eqref{eq:2} and \eqref{eq:3}. Then \emph{Favard's theorem} (see, e.g. \cite[Remark 5.15]{Schmudgen2017}) asserts existence of a probability measure $\mu$ on the Borel subsets of the real line such that the sequence $(p_n : n \geq 0)$ satisfies the formula \eqref{eq:1}. The measure $\mu$ does not need to be unique. If it is unique, the measure $\mu$ is called \emph{determinate}. Otherwise, it is \emph{indeterminate}. If Jacobi parameters are bounded, i.e.
\begin{equation} \label{eq:5}
	\sup_{n \geq 0} (a_n + |b_n|) < \infty,
\end{equation}
then $\mu$ is determinate and $\supp(\mu)$ is compact. More generally, the \emph{Carleman condition}
\begin{equation} \label{eq:6}
	\sum_{n=0}^\infty \frac{1}{a_n} = \infty
\end{equation}
implies determinacy of $\mu$.

A central object in the theory of orthogonal polynomials are the \emph{Christoffel--Darboux kernels} defined for any $n \geq 1$ by
\begin{equation} \label{eq:7}
	K_n(x,y; \mu) = \sum_{j=0}^{n-1} p_j(x) p_j(y), \quad x,y \in \RR.
\end{equation}
We shall omit $\mu$ in the notation if the measure $\mu$ is clear from the context.
The kernels $K_n$ are reproducing in $L^2(\mu)$ for $\PP_{n-1}$, the subspace of all complex polynomials with degree less than or equal to $n-1$. It means that 
\begin{equation} \label{eq:8}
	p(x) = \int_\RR K_n(x,y) p(y) \ud \mu(y), \qquad p \in \PP_{n-1}, \ x \in \RR.
\end{equation}
They can be expressed in terms of the polynomials $p_{n-1}$ and $p_n$ through the \emph{Christoffel--Darboux formula}
\[
	K_n(x,y) = a_{n-1} 
	\begin{cases}
		\frac{p_n(x) p_{n-1}(y) - p_{n-1}(x) p_n(y)}{x-y} & \text{if } x \neq y, \\
		p_n'(x) p_{n-1}(x) - p_{n-1}'(x) p_n(x) & \text{otherwise}.
	\end{cases}
\]
For some applications of the Christoffel--Darboux kernel in orthogonal polynomials, spectral theory, random matrix theory and data analysis, see, e.g. \cite{Nevai1986}, \cite{Simon2007}, \cite{Lubinsky2016} and \cite{Lasserre2022}, respectively.

Nevai in his monograph \cite{Nevai1979} considered in detail properties of orthogonal polynomials with Jacobi parameters satisfying
\begin{equation} \label{eq:9}
	\lim_{n \to \infty} a_n = \frac{1}{2} \quad \text{and} \quad
	\lim_{n \to \infty} b_n = 0.
\end{equation}
He proved in \cite[Theorem 4.2.14]{Nevai1979} that under \eqref{eq:9} the following weak convergence of probability measures
\begin{equation} \label{eq:10}
	\frac{1}{n} K_n(x,x) \ud \mu(x) 
	\xrightarrow{w} 
	\frac{1}{\pi \sqrt{1-x^2}} \mathds{1}_{(-1,1)}(x) \ud \lambda(x)
\end{equation}
holds true,
where by $\mathds{1}_X$ we denote the indicator function of the set $X \subset \RR$ and $\lambda$ is the Lebesgue measure on~$\RR$. Let us recall that a sequence of positive finite measures $(\omega_n : n \geq 1)$ on the real line \emph{converges weakly} to some positive finite measure $\omega_\infty$, $\omega_n \xrightarrow{w} \omega_\infty$, if
\[
	\lim_{n \to \infty} 
	\int_\RR f \ud \omega_n =
	\int_\RR f \ud \omega_\infty, \quad f \in \calC_b(\RR).
\]
Here by $\calC_b(\RR)$ we denote the Banach space of bounded continuous functions on the real line equipped with the supremum norm $\| \cdot \|_\infty$. 

The formula \eqref{eq:10} allows one to approximate the measure $\mu$ provided one understands the pointwise behaviour of $K_n(x,x)$. Along these lines, consider the following decomposition of the measure $\mu$
\begin{equation} \label{eq:11}
	\mu = \muAC + \muS,
\end{equation}
where the measures $\muAC$ and $\muS$ are absolutely continuous and singular with respect to the Lebesgue measure, respectively. Let $\mupAC$ be the derivative of the function $t \mapsto \muAC((-\infty,t])$.
Suppose that $\supp(\mu) = [-1,1]$ and $\mupAC(x) > 0$ for $\lambda$-a.a. $x \in [-1,1]$. Then by Rakhmanov's theorem (see \cite[Theorem 2]{Rakhmanov1983} and also \cite{Denisov2004}) the formula \eqref{eq:9} holds true. Moreover, by \cite[Theorem 12.1]{Nevai1987} it holds
\begin{equation} \label{eq:12}
	\lim_{n \to \infty} 
	\int_{-1}^1 \bigg| \frac{1}{n} K_n(x,x) \mupAC(x) - \frac{1}{\pi \sqrt{1-x^2}} \bigg| \ud \lambda(x) = 0
\end{equation}
(cf. \eqref{eq:10}). Furthermore, if a stronger \emph{Szeg\H{o}'s condition} is satisfied, i.e.
\begin{equation} \label{eq:13}
	\int_{-1}^1 \frac{\log \mupAC(x)}{\sqrt{1-x^2}} \ud \lambda(x) > -\infty,
\end{equation}
then one even has
\begin{equation} \label{eq:14}
	\lim_{n \to \infty} \frac{1}{n} K_n(x,x) = 
	\frac{1}{\mupAC(x)} \frac{1}{\pi \sqrt{1-x^2}}, \quad \text{$\lambda$-a.a. } x \in (-1,1)
\end{equation}
(see \cite[Theorem 5]{Mate1991}). Let us mention that generalizations of \eqref{eq:10}, \eqref{eq:12} and \eqref{eq:14} to \emph{compactly supported} measures $\mu$ satisfying some regularity conditions have been accomplished in \cite{StahlTotik1992}, \cite{Simon2009} and \cite{Totik2000}, respectively.

Since the condition \eqref{eq:13} is quite restrictive and far from being necessary for \eqref{eq:14} to hold, for a~well-understood measure $\mu$ it is natural to consider new measures of the form
\begin{equation} \label{eq:15}
	\ud \mu_g(x) := g(x) \ud \mu(x),
\end{equation}
where $g$ is a positive $\mu$-measurable function such that the measure $\mu_g$ has all moments finite. By normalizing, one can assume that $\mu_g$ is a probability measure. If the behaviour of $K_n(x,x;\mu)$ is known, it seems to be useful to study the expression
\begin{equation} \label{eq:16}
	\frac{K_n(x,x; \mu_g)}{K_n(x,x; \mu)}.
\end{equation}
To understand the behaviour of \eqref{eq:16} Nevai in \cite[Section 6.2]{Nevai1979} considered for any nonnegative $\mu$-measurable function $f$ and any $n \geq 1$ a new function $G_n[f] : \RR \to [0,\infty]$ defined by
\begin{equation} \label{eq:17}
	G_n[f](x) = 
	\frac{1}{K_n(x,x)} 
	\int_\RR K_n^2(x,y) f(y) \ud \mu(y), \quad x \in \RR.
\end{equation}
According to Proposition~\ref{prop:1} the right-hand side of the formula \eqref{eq:17} is a bounded rational function for any $f \in \calL^\infty(\mu)$, so we shall extend the definition \eqref{eq:17} to such functions. Let us recall that the space $\calL^\infty(\mu)$ consists of $\mu$-essentially bounded measurable functions with the essential supremum seminorm $\|\cdot\|_{\calL^\infty(\mu)}$.

Let us recall some central notions of \cite{Breuer2010a}.
\begin{definition}
Let $\mu$ be a probability measure on the real line with infinite support having all moments finite. 
We say that the measure $\mu$ satisfies \emph{Nevai's condition} at some $x \in \RR$ if
\begin{equation} 
	\label{eq:18}
	\lim_{n \to \infty} G_n[f](x) = f(x), \quad f \in \calC_b(\RR).
\end{equation}
If $K$ is a bounded subset of $\RR$ and
\begin{equation}
	\label{eq:19}
	\lim_{n \to \infty} 
	\sup_{x \in K} 
	\big| G_n[f](x) - f(x) \big| = 0, \quad f \in \calC_b(\RR),
\end{equation}
then we say that the measure $\mu$ satisfies Nevai's condition uniformly on $K$.
\end{definition}
The conditions \eqref{eq:18} and \eqref{eq:19} have been introduced by Nevai in \cite[Section 6.2]{Nevai1979}. His remarkable observation (see Corollary~\ref{thm:1}) states that if $g : \RR \to (0, \infty)$ is such that $g,1/g \in \calL^\infty(\mu)$ and $g$ is continuous at some $x \in \RR$, then
\begin{equation} \label{eq:20}
	\lim_{n \to \infty} 
	\frac{K_n(x,x; \mu_g)}{K_n(x,x; \mu)} = 
	\frac{1}{g(x)}
\end{equation}
provided \eqref{eq:18} is satisfied. Moreover, if additionally $g$ is continuous on some compact $K \subset \RR$ and \eqref{eq:19} holds, then the convergence \eqref{eq:20} is uniform on $K$.

Nevai's condition in the case of compactly supported measures $\mu$ have been studied in detail. It turns out that for such measures the condition \eqref{eq:18} is implied by the subexponential growth of the sequence $(p_n(x) : n \geq 0)$, which means
\begin{equation} \label{eq:21}
	\lim_{n \to \infty} \frac{p_{n}^2(x)}{K_n(x,x)} = 0,
\end{equation}
cf. \cite[Theorem 1.2]{Breuer2010a}. This fact was proven by Nevai (see the proof of \cite[Theorem 6.2.2]{Nevai1979}). Next, \cite[Theorem 8.2]{Breuer2010a} describes a large and important class of compactly supported measures that satisfy Nevai's condition uniformly on $\supp(\mu)$. In \cite{Breuer2010a} some examples of compactly supported measures such that Nevai's condition fails for some $x \in \supp(\mu)$ were constructed. In the article \cite{Lubinsky2011} for any measure $\mu$ with compact support and any $f \in \calL^{\infty}(\lambda)$ it was proven that
\[
	\lim_{n \to \infty} 
	\lambda 
	\big(
		\{ x \in \RR : \mupAC(x) > 0, |G_n[f](x) -f(x)| > \epsilon \}
	\big) = 0, \quad \epsilon > 0
\]
and for any $g : \RR \to (0, \infty)$ satisfying $g,1/g \in \calL^\infty(\mu)$
\[
	\lim_{n \to \infty} 
	\lambda 
	\bigg(
		\bigg\{ x \in \RR : \mupAC(x) > 0,
	\bigg|
		\frac{K_n(x,x;\mu_g)}{K_n(x,x;\mu)} - 
		\frac{1}{g(x)}
	\bigg| > \epsilon \bigg\}
	\bigg) = 0, \quad \epsilon > 0.
\]
Let us also mention the article \cite{Breuer2014}, where Nevai's condition found applications for limit theorems of linear statistics of a class of determinantal point processes, called \emph{orthogonal polynomial ensembles}.

The present article is concerned with the case of measures with \emph{unbounded supports}. On page 251 of the article \cite{Breuer2010a} it was suggested that the study of Nevai's condition in this case seems to be an interesting challenge. It was also mentioned that their methods completely break down for the measures from the article \cite{Damanik2007}. It turns out that for these measures Nevai's condition holds locally uniformly on $\supp(\mu) \setminus \{0\}$ (see Section~\ref{sec:4.1.4}). Next, \cite[Theorem 4.5]{Breuer2014} allows one to prove Nevai's condition for \emph{arbitrary measures} at Lebesgue points of $\mu$ under the additional \emph{sine kernel universality} of the Christoffel--Darboux kernels. In view of the recent breakthrough in sine kernel universality (see \cite{Eichinger2021}) this additional hypothesis follows from the previous one.

In Section~\ref{sec:2} we present sufficient conditions for uniform Nevai's condition. Theorem~\ref{thm:2} shows that Nevai's condition always holds on the support of the discrete part of any determinate measure. Interestingly, the same implication holds also for some (the so-called $N$-extremal) indeterminate measures. In Theorem~\ref{thm:3} we show that if a determinate measure $\mu$ is absolutely continuous on some compact $K$ with a continuous positive density, then $\mu$ satisfies Nevai's condition uniformly on $K$. 

Section~\ref{sec:3} is devoted to some implications of uniform Nevai's condition for $\mu$ combined with some information on the behaviour of $(K_n(x,x;\mu) : n \geq 1)$. In Section~\ref{sec:3.2} we study limit theorems for unconventionally normalized global statistics of orthogonal polynomial ensembles. Next, in Section~\ref{sec:3.1}, we relate pointwise, uniform, or weak asymptotics of $(K_n(x,x;\mu_g) : n \geq 1)$ to the same kind of asymptotics of $(K_n(x,x;\mu) : n \geq 1)$ provided $g,1/g \in \calC_b(\RR)$.

In the last Section~\ref{sec:4} we show examples of measures~$\mu$ with unbounded supports defined in terms of Jacobi parameters which satisfy Nevai's condition locally uniformly on $\supp(\mu) \setminus X$, where $X$ is a finite exceptional set. In particular, we cover a large class of periodically modulated and blended Jacobi parameters. 
Since for such Jacobi parameters we know the asymptotic behaviour of Christoffel--Darboux kernels, this allows us to apply the tools from Section~\ref{sec:3} to measures of the form \eqref{eq:15}. This seems to be an important step forward, since in general no understanding of Jacobi parameters of the measures $\mu_g$ is known.

\subsection*{Acknowledgements}
I would like to thank Bartosz Trojan for useful discussions on the topic of the present article and the reviewers whose valuable comments improved the article significantly.

\section{Properties of Nevai's operators} \label{sec:2}

Let us start with a useful consequence of the reproducing property \eqref{eq:8} of the Christoffel--Darboux kernels. Namely, since $K_n(x,\cdot)$ is a polynomial of degree $n-1$ with respect to the second variable, we get by \eqref{eq:8}
\begin{equation} \label{eq:22}
	K_n(x,x) = \int_\RR K_n^2(x,y) \ud \mu(y), \qquad x \in \RR.
\end{equation}

The following proposition implies that the assignment $\calL^\infty(\mu) \ni f \mapsto G_n[f]$ defines a \emph{contractive operator} $G_n : \calL^\infty(\mu) \to \calC_b(\RR)$ (cf. \cite[Properties 6.2.1(iii)]{Nevai1979}). Therefore, it seems natural to call $(G_n : n \geq 1)$ \emph{Nevai's operators}. Let us recall that for any $f : \RR \to \CC$ we define
\[
    \| f \|_{\infty} := \sup_{x \in \RR} |f(x)|.
\]

\begin{proposition} \label{prop:1}
Let $f \in \calL^\infty(\mu)$. Then $G_n[f]$ is a rational function on $\RR$ satisfying
\begin{equation}
    \label{eq:23}
    \sup_{x \in \RR} | G_n[f](x) | = \| G_n[f] \|_{\infty} \leq \| f \|_{\calL^\infty(\mu)}.
\end{equation}
Moreover, if $f(x) \geq 0$ for $\mu$-a.a. $x \in \RR$, then $G_n[f](x) \geq 0$ for any $x \in \RR$.
\end{proposition}
\begin{proof}
Let $f \in \calL^\infty(\mu)$. Then by H\"{o}lder's inequality together with \eqref{eq:22}
\begin{equation} \label{eq:24}
	| G_n[f](x) | \leq 
	\int_\RR \frac{K_n^2(x,y)}{K_n(x,x)} |f(y)| \ud \mu(y)
	\leq \| f \|_{\calL^\infty(\mu)}, \quad x \in \RR,
\end{equation}
which shows that $G_n[f]$ is a well-defined function and it satisfies \eqref{eq:23}.
Next, by \eqref{eq:7}
\begin{equation} \label{eq:25}
	G_n[f](x) =
	\frac{1}{K_n(x,x)} 
	\sum_{0\leq j,j' \leq n-1}
	p_j(x) p_{j'}(x) 
	\int_\RR p_j(y) p_{j'}(y) f(y) \ud \mu(y).
\end{equation}
By H\"{o}lder's inequality applied twice
\begin{align*}
	\bigg| \int_\RR p_j(y) p_{j'}(y) f(y) \ud \mu(y) \bigg|
	&\leq
	\| f \|_{\calL^\infty(\mu)} 
	\int_\RR |p_j(y) p_{j'}(y)| \ud \mu(y) \\
	&\leq
	\| f \|_{\calL^\infty(\mu)} 
	\| p_j \|_{L^2(\mu)} \|p_{j'} \|_{L^2(\mu)} \\
	&= 
	\| f \|_{\calL^\infty(\mu)},
\end{align*}
where in the last equality we used \eqref{eq:1}. Thus, the expression \eqref{eq:25} is a well-defined rational function. Since its denominator does not vanish, it is a continuous function. Finally, the non-negativity of the function $G_n[f]$ for any non-negative $f$ follows immediately from formula \eqref{eq:17}.
\end{proof}

\subsection{Sufficient conditions for Nevai's condition} \label{sec:2.1}

It turns out that Nevai's condition has an alternative definition in the measure theoretical language. Namely, for a given $x \in \RR$ let us define the sequence of measures
\begin{equation}
	\label{eq:26}
	\ud \omega^x_n(y) = \frac{K_n^2(x, y)}{K_n(x, x)} \ud \mu(y).
\end{equation}
By \eqref{eq:22} these measures are probability measures.
Moreover, by \eqref{eq:17} and \eqref{eq:26} one has
\begin{equation}
	\label{eq:27}
	G_n[f](x) = \int_\RR f(y) \ud \omega^x_n(y).
\end{equation}
Hence, the condition \eqref{eq:18} is equivalent to $\omega^x_n \xrightarrow{w} \delta_{x}$, where $\delta_x$ is the Dirac measure at $x$.

We start with the following lemma, which is part of the proof of \cite[Theorem 4.4]{Breuer2014}. 
\begin{lemma} 
    \label{lem:2}
    Let $K$ be a bounded subset of $\RR$. If there exists $\eta_0 > 0$ such that for each $\eta \in (0, \eta_0]$,
    \begin{equation}
        \label{eq:28}
        \liminf_{n \to \infty} \inf_{x \in K} \omega^x_n \big([x-\eta, x+\eta]\big) = 1,
    \end{equation}
    then $\mu$ satisfies Nevai's condition uniformly on $K$.
\end{lemma}
\begin{proof}
	Let $f \in \calC_b(\RR)$. By \eqref{eq:27} we need to show that
	\begin{equation}
		\label{eq:29}
		\lim_{n \to \infty} 
		\sup_{x \in K} 
		\bigg| \int_\RR \big( f(y) - f(x) \big) \ud \omega_n^x(y) \bigg| = 0.
	\end{equation}
	To see this, we write
	\begin{align*}
		\bigg| \int_\RR \big( f(y) - f(x) \big) \ud \omega_n^x(y) \bigg| 
		&\leq
		\int_{\RR \setminus [x-\eta, x+\eta]} |f(y) - f(x)| \ud \omega_n^x(y) \\
		&\phantom{\leq}+
		\int_{[x-\eta, x+\eta]} |f(y) - f(x)| \ud \omega_n^x(y).
	\end{align*}
	For the first integral we have
	\begin{align}
		\nonumber
		\sup_{x \in K}
		\int_{\RR \setminus [x-\eta, x+\eta]} 
		|f(y) - f(x)| \ud \omega_n^x(y) 
		&\leq
		2 \| f \|_\infty \sup_{x \in K} \Big( 1 - \omega_{n}^x\big([x-\eta,x+\eta ]\big) \Big) \\
		&=
		\label{eq:30}
		2 \| f \|_\infty \Big( 1 - \inf_{x \in K} \omega_{n}^x\big([x-\eta,x+\eta]\big) \Big).
	\end{align}
	The second integral can be estimated as follows
	\begin{equation}
		\label{eq:31}
		\sup_{x \in K} 
		\int_{[x-\eta, x+\eta]} |f(y) - f(x)| \ud \omega_n^x(y) 
		\leq
		\sup_{x \in K} \sup_{h \in [-\eta,\eta ]} |f(x+h) - f(x)|.
	\end{equation}
	Now, in view of \eqref{eq:28}, the right-hand side of \eqref{eq:30} tends to $0$ for any $\eta \in (0, \eta_0]$. Next, 
	the right-hand side of \eqref{eq:31} tends to $0$ as $\eta \downarrow 0$ by the uniform continuity of $f$ on the compact set(\footnote{For any $X \subset \RR$ by $\cl{X}$ we denote its closure in Euclidean topology.})
	$\tilde{K} := \cl{K} + [-\delta,\delta]$. Hence, \eqref{eq:29} follows.
\end{proof}

The following theorem implies that Nevai's condition is always satisfied at each point of the support of the discrete part
of any determinate measure~$\mu$.
\begin{theorem} \label{thm:2}
    Let $\mu$ be a probability measure on the real line with all moments finite. Suppose that the corresponding sequence $(p_n : n \geq 0)$ of orthogonal polynomials forms an orthonormal basis(\footnote{In particular, this property is satisfied when the measure $\mu$ is determinate, see, e.g. \cite[Theorem 6.10]{Schmudgen2017}.}) in $L^2(\mu)$. If for some $x \in \RR$ it holds $\mu(\{x\}) > 0$, then the measure $\mu$ satisfies Nevai's condition at $x$.
\end{theorem}
\begin{proof}
	For every $\eta > 0$ we have
	\[
		\omega^x_n\big([x-\eta, x+\eta]\big) 
		\geq  
		\frac{K_n^2(x,x;\mu)}{K_n(x,x;\mu)} \mu(\{x\}) =
		K_n(x,x;\mu) \mu(\{x\}).
	\]
	According to M. Riesz theorem (see, e.g. \cite[Theorem 2.14]{Shohat1943}) the orthogonal polynomials $(p_n : n \geq 0)$ form an orthonormal basis in $L^2(\mu)$ if and only if the measure $\mu$ is either determinate or the so-called $N$-extremal. Then
	\[
		\lim_{n \to \infty} K_n(x,x;\mu) = \frac{1}{\mu(\{x\})},
	\]
	(see, e.g. \cite[Corollary 2.6]{Shohat1943} in the determinate case and \cite[Theorem 2.13]{Shohat1943} in the N-extremal case) and the result follows from Lemma~\ref{lem:2} applied to $K = \{x\}$.
\end{proof}

By adapting the proofs of \cite[Theorems 4.4 and 4.5]{Breuer2014} and using \cite[Theorem 1.4]{Eichinger2021}, we obtain the following theorem.
\begin{theorem} 
	\label{thm:3}
	Let $\mu$ be a determinate probability measure on the real line. 
If the measure $\mu$ is absolutely continuous on some open $U \subset \RR$ with positive and continuous density, then $\mu$ satisfies Nevai's condition uniformly on every compact $K \subset U$.
\end{theorem}
\begin{proof}
Let $K \subset U$ be compact. 
We are going to prove that the hypotheses of Lemma~\ref{lem:2} are satisfied on $K$. Let $\eta_0 > 0$ be such that 
\[
    \bigcup_{x \in K} [x -\eta_0, x + \eta_0] \subset U
\]

Let $\lambda_n(x) = 1/K_n(x,x)$ be the Christoffel function. Since $\mu$ is determinate and $\mu(\{x\}) = 0$ for any $x \in K$, we have
\[
    \lim_{n \to \infty} \lambda_n(x) = 0, \quad x \in K.
\]
Since $(\lambda_n(x) : n \geq 1)$ is a sequence of decreasing continuous functions, by Dini's theorem (see, e.g. \cite[Theorem 7.13]{Rudin1976}), the convergence is uniform on $K$. Therefore,
\begin{equation} \label{eq:93}
	\lim_{n \to \infty} \inf_{x \in K} K_n(x,x) = \infty.
\end{equation}
Define
\begin{equation} \label{eq:94}
	\tau_n(x) = \mupAC(x) K_n(x,x).
\end{equation}
Since $\mupAC$ is continuous and positive on $K$, the formula \eqref{eq:93} leads to
\begin{equation} \label{eq:95}
	\lim_{n \to \infty} \inf_{x \in K} \tau_n(x) = \infty.
\end{equation}
To prove \eqref{eq:28}, we observe that for $\eta \in (0, \eta_0]$ by \eqref{eq:94}
\begin{align*}
    \omega^x_n\big( [x-\eta, x+\eta] \big) 
    &= 
    \int_{x -\eta}^{x+\eta} 
    \frac{K_n^2(x, y)}{K_n(x, x)} \mupAC(y) \ud y \\
    &= 
    \int_{-\eta \tau_n(x)}^{\eta \tau_n(x)} 
    \frac{K_n^2 \Big(x, x+\tfrac{t}{\tau_n(x)} \Big)}{K_n^2(x,x)}
    \frac{\mupAC \Big(x+\tfrac{t}{\tau_n(x)} \Big)}{\mupAC(x)} \ud t.
\end{align*}
By \eqref{eq:95} for any $M > 0$ there exists $N_0 > 0$ such that for any $n \geq N_0$ we have $\inf_{x \in K} \tau_n(x) \geq M/\eta$. 
Thus, 
\[
    \omega^x_n\big( [x-\eta, x+\eta] \big) 
    \geq 
    \int_{-M}^{M} 
    \frac{K_n^2 \Big(x, x+\tfrac{t}{\tau_n(x)} \Big)}{K_n^2(x,x)}
    \frac{\mupAC \Big(x+\tfrac{t}{\tau_n(x)} \Big)}{\mupAC(x)} \ud t.
\]
Since $\mupAC$ is continuous and positive on $K$, we have by \cite[Theorem 1.4]{Eichinger2021} and \cite[Lemma 3.3]{Garnett2007}
\[
    \lim_{n \to \infty}
    \frac{K_n^2 \Big(x, x+\tfrac{t}{\tau_n(x)} \Big)}{K_n^2(x,x)}
    \frac{\mupAC \Big(x+\tfrac{t}{\tau_n(x)} \Big)}{\mupAC(x)} 
    =
    \sinc^2(\pi t)
\]
uniformly(\footnote{Let us recall that 
$\sinc(x) =
    \begin{cases}
        \frac{\sin(x)}{x} & x \neq 0, \\
        1 & \text{otherwise}.
    \end{cases}$}) with respect to $(x,t) \in K \times [-M, M]$. Thus, 
\begin{align*}
    \liminf_{n \to \infty}
    \inf_{x \in K} \omega^x_n\big( [x-\eta, x+\eta] \big) 
    &\geq 
    \int_{-M}^{M} 
    \sinc^2(\pi t) \ud t.
\end{align*}
Since $M > 0$ was arbitrary and
\[
    \int_{-\infty}^\infty \sinc^2(\pi t) \ud t = 1
\]
we get \eqref{eq:28}, which is what we needed to show.
\end{proof}

\section{Applications of Nevai's condition} \label{sec:3}
In this section we shall study implications of uniform Nevai's condition for $\mu$ combined with some information on the behaviour of $(K_n(x,x;\mu) : n \geq 1)$.

\subsection{Law of large numbers for global linear statistics of orthogonal polynomial ensembles} \label{sec:3.2}
In this section we shall study limit theorems for linear statistics of an important subclass of determinantal point processes, which contains many models considered in physics, statistical mechanics, probability theory and combinatorics, see the survey \cite{Konig2005} for more details.

An \emph{orthogonal polynomial ensemble} of size $n \in \NN$ is a probability measure $\PP^{(n)}$ on $\RR^n$ of the form
\begin{equation} \label{eq:43}
	\ud \PP^{(n)} (\lambda_1,\ldots,\lambda_n) =
	\frac{1}{Z_n} 
	\prod_{1 \leq i < j \leq n} (\lambda_j - \lambda_i)^2 
	\ud \mu(\lambda_1) \ldots \ud \mu(\lambda_n),
\end{equation}
where $\mu$ is a probability measure on $\RR$ with all moments finite and $Z_n > 0$ is the normalizing constant. It can be shown that \eqref{eq:43} can be rewritten to
\begin{equation} \label{eq:44}
	\ud \PP^{(n)} (\lambda_1,\ldots,\lambda_n) =
	\frac{1}{n!}
	\det \big[ K_n(\lambda_i, \lambda_j; \mu) \big]_{i,j=1}^n
	\ud \mu(\lambda_1) \ldots \ud \mu(\lambda_n),
\end{equation}
see, e.g. \cite[Lemma 2.8]{Konig2005}. 

Suppose that a random vector $\big\{ \lambda_i^{(n)} \big\}_{i=1}^n$ has the distribution $\PP^{(n)}$. One can associate to it a random measure
\[
	\Xi^{(n)} = \sum_{i=1}^n \delta_{\lambda_i^{(n)}}.
\]
For any $\mu$-measurable function $f : \RR \to \RR$ one considers a \emph{linear statistic} of $\Xi^{(n)}$ defined by
\[
	\Xi^{(n)} f = \sum_{i=1}^n f \big( \lambda_i^{(n)} \big).
\]
From the determinantal structure of \eqref{eq:44} we have
\begin{equation}
	\label{eq:45}
	\EE [\Xi^{(n)} f] =
	\int_\RR f(x) K_n(x,x) \ud \mu(x)
\end{equation}
and
\begin{equation}
	\label{eq:46}
	\Var[\Xi^{(n)} f] =
	\int_\RR f^2(x) K_n(x,x) \ud \mu(x) -
	\int_\RR \int_\RR f(x) f(y) K_n^2(x,y) \ud \mu(x) \ud \mu(y),
\end{equation}
see, e.g. \cite[Appendix A.2]{Breuer2014}.

In what follows, we shall assume that the random measures $(\Xi^{(n)} : n \geq 1)$ are defined on the same probability space $(\Omega, \calF, \PP)$. 

In the following lemma we assume that we understand weak asympotic behavior of the sequence $(K_n(x,x) : n \geq 1)$ (see \eqref{eq:47}). Recall that for compactly supported measures, the right choice is usually $w(x) \equiv 1$ and $\rho_n = n$ (see, e.g. \cite{Simon2009}). In Section~\ref{sec:4.1} we show examples of measures with unbounded supports such that one needs to take $w(x)=1+x^2$ and a more general sequence $(\rho_n : n \geq 1)$ (see formulas \eqref{eq:70}, \eqref{eq:77}, and \eqref{eq:85}).
\begin{lemma} \label{lem:3}
Suppose that for some positive sequence $(\rho_n : n \geq 1)$ tending to infinity, there are a continuous function $w:\RR \to [1,\infty)$ and a positive Radon measure $\nu$ satisfying $\int_\RR \frac{1}{w(x)} \ud \nu(x) < \infty$ such that
\begin{equation}
	\label{eq:47}
	\frac{1}{\rho_n} 
	\frac{K_n(x,x)}{w(x)} \ud \mu(x) 
	\xrightarrow{w} 
	\frac{1}{w(x)} \ud \nu(x).
\end{equation}
Then for any $f \in \calC_b(\RR)$ satisfying $\sup_{x \in \RR} w(x) |f(x)| < \infty$ we have
\begin{equation}
	\label{eq:48}
	\lim_{n \to \infty} 
	\frac{1}{\rho_n} 
	\EE[\Xi^{(n)} f] =
	\int_\RR f \ud \nu.
\end{equation}
Moreover, for any $\epsilon > 0$
\begin{equation}
	\label{eq:50}
	\lim_{n \to \infty} 
	\PP 
	\bigg( 
	\bigg|
		\frac{1}{\rho_n} \Xi^{(n)} f - \int_\RR f \ud \nu
	\bigg| \geq \epsilon 
	\bigg) 
	= 
	0.
\end{equation}
\end{lemma}
\begin{proof}
The formula \eqref{eq:48} follows immediately from \eqref{eq:45} and \eqref{eq:47}. Next, in view of \eqref{eq:17}
\begin{equation} \label{eq:51}
    \int_\RR \int_\RR f(x) f(y) K_n^2(x,y) \ud \mu(x) \ud \mu(y) 
    =
    \int_\RR 
    f(x) 
    G_n[f](x)
    K_n(x,x) \ud \mu(x).
\end{equation}
Notice
\begin{align}
    \nonumber
    \bigg|
    \int_\RR 
    f(x) 
    G_n[f](x)
    K_n(x,x) \ud \mu(x)
    \bigg|
    &\leq 
    \sup_{y \in \RR} |G_n[f](y)|
    \int_\RR 
    |f(x)|
    K_n(x,x) \ud \mu(x) \\
    \label{eq:52}
    &\leq
    \norm{f}_\infty
    \int_\RR
    |f(x)| K_n(x,x) \ud \mu(x),
\end{align}
where in the last inequality we used Proposition~\ref{prop:1}.
Therefore, by \eqref{eq:46} and \eqref{eq:47} we get
\begin{equation} \label{eq:49}
	\limsup_{n \to \infty} 
	\frac{1}{\rho_n} \Var [\Xi^{(n)} f]
	\leq
	\int_\RR f^2 \ud \nu +
	\norm{f}_\infty \int_\RR |f| \ud \nu < \infty.
\end{equation}
Finally, by Chebyshev's inequality, it implies that for any $\epsilon > 0$
\begin{align}
	\nonumber
	\PP 
	\bigg( 
	\bigg|
		\frac{1}{\rho_n} \Xi^{(n)} f - 
		\EE \bigg[ \frac{1}{\rho_n} \Xi^{(n)} f \bigg] 
	\bigg| \geq \epsilon 
	\bigg) 
	&\leq
	\frac{1}{\epsilon^2} \Var \bigg[ \frac{1}{\rho_n} \Xi^{(n)} f \bigg] \\
	\label{eq:53}
	&\leq
	\frac{1}{\epsilon^2} \frac{1}{\rho_n^2} \Var[\Xi^{(n)} f].
\end{align}
Thus, the formula \eqref{eq:50} follows from \eqref{eq:49} and \eqref{eq:48}.
\end{proof}

Under some additional hypotheses we can improve the convergence \eqref{eq:50} to $\PP$-almost sure convergence.
\begin{theorem} \label{thm:4}
Suppose that for some strictly increasing sequence $(\rho_n : n \geq 1)$ tending to infinity, there are a continuous function $w:\RR \to [1,\infty)$ and a positive Radon measure $\nu$ satisfying $\int_\RR \frac{1}{w(x)} \ud \nu(x) < \infty$ such that
\begin{equation}
	\label{eq:47'}
	\frac{1}{\rho_n} 
	\frac{K_n(x,x)}{w(x)} \ud \mu(x) 
	\xrightarrow{w} 
	\frac{1}{w(x)} \ud \nu(x).
\end{equation}
Suppose that there is $X \subset \RR$ with $\nu(X) = 0$ such that $\mu$ satisfies Nevai's condition locally uniformly around each $x \in \RR \setminus X$. If 
\begin{equation} \label{eq:98}
	\lim_{n \to \infty} n (\rho_{n+1} - \rho_n) = \infty,
\end{equation}
then for any $f \in \calC_b(\RR)$ satisfying $\sup_{x \in \RR} w(x) |f(x)| < \infty$ we have
\begin{equation} \label{eq:60'}
	\PP
	\bigg(
	\lim_{n \to \infty}
	\frac{1}{\rho_n} \Xi^{(n)} f = \int_\RR f \ud \nu
	\bigg)
	= 1.
\end{equation}
\end{theorem}
\begin{proof}
Since for any $\epsilon > 0$
\[
	\PP 
	\bigg( 
	\bigg|
		\frac{1}{\rho_n} \Xi^{(n)} f - 
		\frac{1}{\rho_n} \EE[\Xi^{(n)} f]
	\bigg| \geq \epsilon 
	\bigg)
	=
	\PP 
	\big( 
	\big|
		\Xi^{(n)} f - 
		\EE[\Xi^{(n)} f]
	\big| \geq \epsilon \rho_n
	\big).
\]
We have by \cite[Theorem 2.1]{Breuer2014} the bound 
\begin{equation} \label{eq:59'}
	\PP 
	\bigg( 
	\bigg|
		\frac{1}{\rho_n} \Xi^{(n)} f - 
		\frac{1}{\rho_n} \EE[\Xi^{(n)} f]
	\bigg| \geq \epsilon 
	\bigg)
	\leq
	2 \exp \bigg(-\frac{\epsilon}{6 \norm{f}_\infty} \rho_n \bigg).
\end{equation}
for all $n \geq 1$ such that
\begin{equation} \label{eq:62}
	\epsilon \geq 
	\frac{2A}{3 \norm{f}_\infty} 
	\frac{1}{\rho_n} \Var[\Xi^{(n)} f],
\end{equation}
where $A > 0$ is some constant. If we prove that
\begin{equation} \label{eq:55}
	\lim_{n \to \infty} \frac{1}{\rho_n} \Var[\Xi^{(n)} f] = 0.
\end{equation}
Then the right-hand side of \eqref{eq:62} tends to $0$. So there is a constant $M \geq 1$ such that the inequality \eqref{eq:62} is satisfied for all $n \geq M$, which implies that \eqref{eq:59'} holds for all $n \geq M$. Therefore, by the Borel--Cantelli lemma, we get \eqref{eq:60'} provided
\begin{equation} \label{eq:61'}
	\sum_{n=1}^\infty \ue^{-c \rho_n} < \infty \quad \text{for any } c>0.
\end{equation}
To verify \eqref{eq:61'} we shall use Raabe's test. Let $r_n = \rho_{n+1} - \rho_n$. Notice that
\[
	\frac{1}{c r_n} 
	\bigg( \frac{\ue^{-c \rho_n}}{\ue^{-c \rho_{n+1}}} - 1 \bigg) 
	=
	\frac{\ue^{c r_n} - 1}{c r_n} \geq 1.
\]
Thus, by \eqref{eq:98}
\[
	\liminf_{n \to \infty}
	n \bigg( \frac{\ue^{-c \rho_n}}{\ue^{-c \rho_{n+1}}} - 1 \bigg)
	\geq c \liminf_{n \to \infty} n r_n = \infty.
\]
Thus, Raabe's test implies \eqref{eq:61'}.

It remains to prove \eqref{eq:55}. By \eqref{eq:46} and \eqref{eq:17} we can write
\begin{equation} \label{eq:56}
	\frac{1}{\rho_n} \Var[\Xi^{(n)} f] =
	\frac{1}{\rho_n} \int_\RR f^2(x) K_n(x,x) \ud \mu -
	\int_\RR h_n \ud \eta_n
\end{equation}
where
\begin{equation} \label{eq:57}
	h_n(x) =
	w(x) f(x) G_n[f](x), \quad
	\ud \eta_n = \frac{1}{\rho_n} \frac{K_n(x,x)}{w(x)} \ud \mu(x)
\end{equation}
Since $\Var[\Xi^{(n)} f] \geq 0$ by taking the limit superior and using \eqref{eq:48} we get
\[
	\limsup_{n \to \infty} \int_\RR h_n \ud \eta_n
	\leq 
	\int_\RR f^2 \ud \nu.
\]
Thus, \eqref{eq:55} will follow once we prove
\begin{equation} \label{eq:58}
	\liminf_{n \to \infty} 
	\int_\RR h_n \ud \eta_n \geq 
	\int_\RR f^2 \ud \nu.
\end{equation}
We are going to apply Fatou's lemma for weakly converging measures (see \cite[Theorem 2.2]{Feinberg2020}). Notice that \eqref{eq:47'} states that $\eta_n \xrightarrow{w} \frac{1}{w(x)} \ud \nu(x)$. 
Next, since
\[
	\norm{h_n}_\infty \leq 
	\norm{f}_\infty \cdot \sup_{x \in \RR} w(x) |f(x)| < \infty,
\]
the sequence $(h_n : n \geq 1)$ is uniformly integrable with respect to $(\eta_n : n \geq 1)$ (cf. \cite[Definition 2.4]{Feinberg2020}). Therefore, the hypotheses of \cite[Theorem 2.2]{Feinberg2020} are satisfied. Therefore,
\begin{equation} \label{eq:96}
	\liminf_{n \to \infty} 
	\int_\RR h_n(x) \ud \eta_n(x)
	\geq
	\int_\RR \liminf_{\substack{n \to \infty\\x' \to x}} 
	h_n(x') \frac{1}{w(x)} \ud \nu(x).
\end{equation}

\begin{claim} \label{clm:1}
If $\mu$ satisfies Nevai's condition uniformly in some compact interval $K \subset \RR$, then for any $f \in \calC_b(\RR)$
\begin{equation} \label{eq:97}
	\lim_{\substack{n \to \infty\\x' \to x}} G_n[f](x') = f(x)
\end{equation}
holds for any $x \in \intr(K)$. 
\end{claim}
To see this, it is enough to notice that for any $x,x' \in K$
\[
	\big| G_n[f](x') - f(x) \big| \leq
	\sup_{x' \in K}
	\big| G_n[f](x') - f(x') \big| + 
	\big| f(x') - f(x) \big|
\]
we have by \eqref{eq:19}
\[
	\lim_{\substack{n \to \infty\\x' \to x}}
	\big| G_n[f](x') - f(x) \big| 
	\leq 
	\lim_{x' \to x} |f(x') - f(x)| = 0,
\]
by continuity of $f$ at $x$.

Since by the assumption Nevai's condition holds locally uniformly around $\nu$-a.a. $x \in \RR$, by using Claim~\ref{clm:1} in \eqref{eq:96} and recalling \eqref{eq:57} we get \eqref{eq:58}, and consequently, also \eqref{eq:55}. The proof is complete.
\end{proof}

The following result shows that if we strengthen the condition~\eqref{eq:47'}, then we can substantially weaken the required form of Nevai's condition and study linear statistics associated with compactly supported bounded test functions.
\begin{theorem} \label{thm:6}
Let $K \subset \RR$ be compact. Suppose that for a strictly increasing sequence $(\rho_n : n \geq 1)$ tending to infinity, there is a positive Radon measure $\nu$ such that for every $f \in \calL^\infty(K, \mu)$, a bounded $\mu$-measurable function supported on $K$,
\begin{equation} \label{eq:100}
	\lim_{n \to \infty} 
	\frac{1}{\rho_n} \int_\RR f(x) K_n(x,x) \ud \mu(x) =
	\int_\RR f \ud \nu.
\end{equation}
Suppose that for any $\epsilon > 0$ and any $f \in \calL^\infty(K, \mu)$
\begin{equation} \label{eq:101}
	\lim_{n \to \infty}
	\nu \big( \{ x \in K : |G_n[f](x) - f(x)| > \epsilon \}\big) = 0.
\end{equation}
If
\[
	\lim_{n \to \infty} n (\rho_{n+1} - \rho_n) = \infty,
\]
then for any $f \in \calL^\infty(K, \mu)$ we have
\[
	\PP
	\bigg(
	\lim_{n \to \infty}
	\frac{1}{\rho_n} \Xi^{(n)} f = \int_\RR f \ud \nu
	\bigg)
	= 1.
\]
\end{theorem}
\begin{proof}
Notice that \eqref{eq:100} is stronger than \eqref{eq:47'} when applied to functions supported on $K$.

Let $f \in \calL^\infty(K, \mu)$.
By an analogous reasoning to the proof of Theorem~\ref{thm:4} we need only to prove that 
\begin{equation} \label{eq:102}
	\liminf_{n \to \infty} 
	\int_\RR h_n \ud \nu_n
	\geq 
	\int_\RR f^2 \ud \nu
\end{equation}
(cf. \eqref{eq:58}), where 
\[
	h_n(x) = f(x) G_n[f](x), \quad
	\ud \eta_n = \frac{1}{\rho_n} K_n(x,x) \ud \mu(x)
\]
(cf. \eqref{eq:57}). We are going to apply Fatou's lemma for setwise converging measures (see \cite[Theorem 4.2]{Feinberg2020}. Notice that \eqref{eq:100} implies that $\eta_n(\cdot \cap K) \to \nu(\cdot \cap K)$ setwise (cf. \cite[Definition 2.2]{Feinberg2020}). Next, since
\[
    \sup_{n \geq 1} \| h_n \|_{\calL^\infty(\mu)} 
    \leq 
    \| f \|_{\calL^\infty(\mu)}^2 < \infty
\]
the sequence $(h_n : n \geq 1)$ is uniformly integrable with respect to $(\eta_n : n \geq 1)$. In view of \eqref{eq:101} the hypotheses of \cite[Theorem 4.2]{Feinberg2020} are satisfied, and consequently, we obtain \eqref{eq:102}. The proof is complete.
\end{proof}

The following example should be compared with \cite[Remark 3.6]{Breuer2014}.
\begin{example}
Let $\mu$ be a measure on the real line with support $K$ consisting of finitely many compact intervals. Suppose that $\mupAC(x) > 0$ for $\lambda$-a.a. $x \in K$. Let $\rho_n = n$. Then by \cite[Corollary 1]{Totik2000}, the hypothesis \eqref{eq:100} is satisfied, where $\nu$ is the equilibrium measure of $K$. Recall that $\nu$ is absolutely continuous on $K$ with a density $\nu'$ which is real analytic on $\intr{K}$ and belongs to $\calL^p(K, \lambda)$ for any $1<p<2$, see, e.g. \cite[Theorem IV.2.1]{SaffTotik2024}. Let $1<p<2$, then for $q = p/(p-1)$ by H\"older's inequality
\[
	\| G_n[f] - f \|_{\calL^1(K, \nu)} \leq 
	\| G_n[f] - f \|_{\calL^q(K, \lambda)} \cdot
	\| \nu' \|_{\calL^p(K, \lambda)}.
\]
Thus, by \cite[Theorem 1.2]{Lubinsky2011}, we get
\[
	\lim_{n \to \infty} \| G_n[f] - f \|_{\calL^1(K, \nu)} = 0,
\]
and the condition~\eqref{eq:101} is implied by Chebyshev's inequality. Let us recall that for the case $K = [-1,1]$ the condition~\eqref{eq:100} is implied by \cite[Theorem 2]{Mate1984} and condition \eqref{eq:55} directly follows from \cite[Theorem 2]{Nevai1980}.
\end{example}

\subsection{Comparative asymptotics of Christoffel--Darboux kernels on the diagonal} \label{sec:3.1}

Let $\mu$ be a probability measure on $\RR$ with infinite support and all moments finite.
Let $g : \RR \to (0, \infty)$ be such that $g,1/g \in \calL^\infty(\mu)$. Let the measure $\mu_g$ be defined as in \eqref{eq:15}.
Then by \cite[p. 76, Theorem 6.2.3]{Nevai1979} we have 	
\begin{equation} \label{eq:37}
    \frac{1}{G_n[g](x)} \leq 
    \frac{K_n(x,x; \mu_g)}{K_n(x,x; \mu)} \leq 
    G_n \bigg[ \frac{1}{g} \bigg](x), \quad x \in \RR.
\end{equation}
Thus,
\[
    \inf_{y \in \RR} \frac{1}{G_n[g](y)} \leq 
    \frac{K_n(x,x; \mu_g)}{K_n(x,x; \mu)} \leq 
    \sup_{y \in \RR} G_n \bigg[ \frac{1}{g} \bigg](y), \quad x \in \RR,
\]
so by applying Proposition~\ref{prop:1} we get
\begin{equation} \label{eq:38}
    \frac{1}{\|g\|_{\calL^\infty(\mu)}} \leq 
    \frac{K_n(x,x; \mu_g)}{K_n(x,x; \mu)} \leq 
    \bigg\| \frac{1}{g} \bigg\|_{\calL^\infty(\mu)}, \quad x \in \RR.
\end{equation}
\begin{corollary} 
	\label{thm:1}
	Suppose that $\mu$ satisfies Nevai's condition at some $x \in \RR$.
	Let $g : \RR \to (0, \infty)$ be such that $g,1/g \in \calL^\infty(\mu)$. 
	If $g$ is continuous at $x$, then
	\begin{equation} 
		\label{eq:39}
		\lim_{n \to \infty} \frac{K_n(x,x; \mu_g)}{K_n(x,x;\mu)} = \frac{1}{g(x)}
	\end{equation}
	Moreover, if $\mu$ satisfies uniform Nevai's condition on some compact $K \subset \RR$
	and $g$ is continuous on $K$, then the convergence in \eqref{eq:39} is uniform on~$K$.
\end{corollary}
\begin{proof}
It follows from applying \eqref{eq:18} (or \eqref{eq:19}) to \eqref{eq:37} and using \cite[Theorem 13.16(iii)]{Klenke2020}.
\end{proof}

Recall that in Section~\ref{sec:3.2} a central assumption was the property~\eqref{eq:47}. 
In Section~\ref{sec:4.1} we show examples of measures, defined in terms of their Jacobi parameters, for which the condition~\eqref{eq:47} holds true. The next theorem, under some form of Nevai's condition for a measure $\mu$, allows one to extend the convergence~\eqref{eq:47} to measures $\mu_g$ for any $g : \RR \to (0, \infty)$ such that $g,1/g \in \calC_b(\RR)$.

\begin{theorem} \label{thm:5}
Suppose that for some positive sequence $(\rho_n : n \geq 1)$ tending to infinity, there are a continuous function $w : \RR \to [1, \infty)$ and a positive Radon measure $\nu$ satisfying $\int_\RR \frac{1}{w(x)} \ud \nu(x) < \infty$ such that
\begin{equation} 
    \label{eq:40}
    \frac{1}{\rho_n} \frac{K_n(x,x;\mu)}{w(x)} \ud \mu(x) \xrightarrow{w} \frac{1}{w(x)} \ud \nu(x).
\end{equation}
Let $g : \RR \to (0, \infty)$ be such that $g,1/g \in \calC_b(\RR)$.
If there is $X \subset \RR$ with $\nu(X) = 0$ such that $\mu$ satisfies Nevai's condition locally uniformly around each $x \in \RR \setminus X$, then
\begin{equation} \label{eq:41}
	\frac{1}{\rho_n} \frac{K_n(x,x; \mu_g)}{w(x)} \ud \mu_g(x) \xrightarrow{w} \frac{1}{w(x)} \ud \nu(x).
\end{equation}
\end{theorem}
\begin{proof}
Let $f \in \calC_b(\RR)$. We can write
\begin{equation} \label{eq:42}
	\frac{1}{\rho_n}
	\int_\RR f(x) \frac{K_n(x,x; \mu_g)}{w(x)} \ud \mu_g(x) =
	\int_\RR h_n(x) \ud \eta_n(x),
\end{equation}
where
\[
	h_n(x) = f(x) g(x) \frac{K_n(x,x; \mu_g)}{K_n(x,x;\mu)}, \quad
	\ud \eta_n =
	\frac{1}{\rho_n} \frac{K_n(x,x;\mu)}{w(x)} \ud \mu(x).
\]
We are going to apply Lebesgue's convergence theorem for weakly converging measures (see \cite[Corollary 5.1]{Feinberg2020}). Notice that \eqref{eq:40} states that $\eta_n \xrightarrow{w} \tfrac{1}{w(x)} \ud \nu(x)$. Next, since by \eqref{eq:38}
\[
	\| h_n \|_{\infty} \leq \| f \|_{\infty} \|g\|_{\infty} \bigg\| \frac{1}{g} \bigg\|_{\infty} < \infty
\]
the sequence $(h_n : n \geq 1)$ is uniformly integrable with respect to $(\eta_n : n \geq 1)$ (cf. \cite[Definition 2.4]{Feinberg2020}). By Corollary~\ref{thm:1}
\[
	\lim_{n \to \infty} h_n(x) = f(x)
\]
locally uniformly with respect to $x \in \RR \setminus X$. Thus
\[
	\lim_{\substack{n \to \infty\\x' \to x}} h_n(x') = f(x), \quad x \in \RR \setminus X.
\]
Therefore, the hypotheses of \cite[Corollary 5.1]{Feinberg2020} are satisfied, and consequently,
\[
	\lim_{n \to \infty} \int_\RR h_n(x) \ud \eta_n(x) =
	\int_\RR f(x) \frac{1}{w(x)} \ud \nu(x).
\]
In view \eqref{eq:42} it implies \eqref{eq:41}. The proof is complete.
\end{proof}

\section{Examples} \label{sec:4}

In this section, we shall consider unbounded Jacobi parameters defined in terms of periodic sequences. It turns out that it will be useful to impose some regularity conditions. More specifically, we say that a sequence $(x_n : n \geq 1)$ belongs to $\calD_r$ for some integer $r \geq 1$, if it is \emph{bounded} and for each $j \in \{1,\ldots,r \}$
\[
	\sum_{n=1}^\infty |\Delta^j x_n|^{\tfrac{r}{j}} < \infty
\]
where
\begin{align*}
	\Delta^0 x_n &= x_n \\
	\Delta^j x_n &= \Delta^{j-1} x_{n+1} - \Delta^{j-1} x_n, \quad j \geq 1.
\end{align*}
This class was introduced in \cite{Stolz1994}. Let us recall that $\calD_1$ consists of sequences of bounded variation and for any $r \geq 1$ one has $\calD_r \subsetneq \calD_{r+1}$. Given an integer $N \geq 1$ we say that $(x_n : n \geq 1) \in \calD_r^N$ if for any $i=0,1,\ldots,N-1$ we have $(x_{nN+i} : n \geq 1) \in \calD_r$.

It will be useful to rewrite the recurrence relation \eqref{eq:3} in the form
\begin{equation} \label{eq:63}
	\begin{pmatrix}
		p_n(x) \\
		p_{n+1}(x)
	\end{pmatrix}
	= 
	B_n(x)
	\begin{pmatrix}
		p_{n-1}(x) \\
		p_n(x)
	\end{pmatrix}, \quad \text{where} \quad
	B_n(x) =
	\begin{pmatrix}
		0 & 1 \\
		-\frac{a_{n-1}}{a_n} & \frac{x-b_n}{a_n}
	\end{pmatrix}, \quad n \geq 1.
\end{equation}
The matrix $B_n(x)$ is usually called ($1$-step) \emph{transfer matrix}. 

By $(\alpha_n : n \in \ZZ),(\beta_n : n \in \ZZ)$ we denote $N$-periodic sequences of positive and real numbers, respectively. Moreover, we define
\begin{equation} \label{eq:64}
	\frakB_n(x) =
	\begin{pmatrix}
		0 & 1 \\
		-\frac{\alpha_{n-1}}{\alpha_n} & \frac{x-\beta_n}{\alpha_n}
	\end{pmatrix}, \quad n \in \ZZ.
\end{equation}

\subsection{Periodic modulations} \label{sec:4.1}
We say that Jacobi parameters $(a_n),(b_n)$ are \emph{$N$-periodically modulated}, if there exist $N$-periodic sequences $(\alpha_n),(\beta_n)$ of positive and real numbers, respectively, such that
\begin{equation} \label{eq:65}
	\lim_{n \to \infty} 
	\bigg| \frac{a_{n-1}}{a_n} - \frac{\alpha_{n-1}}{\alpha_n} \bigg| = 0, \quad 
	\lim_{n \to \infty} 
	\bigg| \frac{b_n}{a_n} - \frac{\beta_n}{\alpha_n} \bigg| = 0, \quad
	\lim_{n \to \infty} a_n = \infty.
\end{equation}
Next, we define the $N$-step transfer matrix by
\[
	X_n(x) = B_{N+n-1}(x) \ldots B_{n+1}(x) B_n(x), \quad n \geq 1.
\]
Notice that \eqref{eq:65} implies that
\[
	\lim_{j \to \infty} X_{jN+i}(x) = \frakX_i(0), \quad i=0,1,\ldots,N-1,
\]
where
\begin{equation} \label{eq:66}
	\frakX_n(x) = \frakB_{N+n-1}(x) \ldots \frakB_{n+1}(x) \frakB_n(x), \quad n \in \ZZ.
\end{equation}
and $\frakB_n(x)$ is defined in \eqref{eq:64}. This class was introduced in \cite{JanasNaboko2002} and it has been studied ever since (see, e.g. \cite{ChristoffelI, ChristoffelII, Discrete, zeros, jordan, jordan2, PeriodicI, PeriodicII, PeriodicIII, SwiderskiTrojan2019, Simonov2007, Damanik2007, JanasNabokoStolz2004, Naboko2009, Pchelintseva2008, Breuer2010} and the references therein). Notice that by \eqref{eq:65} the condition \eqref{eq:5} is violated. Therefore, measures $\mu$ corresponding to $N$-periodically modulated Jacobi parameters have unbounded supports. It turns out that, \emph{under some regularity assumptions} on Jacobi parameters, some of fundamental properties of $\mu$ depend crucially on the matrix $\frakX_0(0)$. More specifically, one can distinguish the four cases:
\begin{enumerate}[label=\rm (\Roman*), start=1, ref=\Roman*]
\item \label{eq:PI}
if $\tr \frakX_0(0) \in (-2,2)$, then the measure $\mu$ is purely absolutely continuous on $\RR$ with a positive continuous density;

\item if $\tr \frakX_0(0) \in \{-2,2\}$, then we have two subcases:
\begin{enumerate}[label=\rm (\alph*), start=1, ref=II(\alph*)]
\item \label{eq:PIIa}
if $\frakX_0(0)$ is diagonalizable, then there is a compact interval $K \subset \RR$ such that the measure $\mu$ is purely absolutely continuous on $\RR \setminus K$ with a positive continuous density and it is purely discrete in the interior of $K$;

\item \label{eq:PIIb}
if $\frakX_0(0)$ is \emph{not} diagonalizable, then there is a possibly degenerated infinite open interval $\Lambda \subset \RR$ such that the measure $\mu$ is purely absolutely continuous on $\Lambda$ with a positive continuous density and it is purely discrete on $\RR \setminus \Lambda$;
\end{enumerate}

\item \label{eq:PIII}
if $\tr \frakX_0(0) \in \RR \setminus [-2,2]$, then the measure $\mu$ is purely discrete with the support having no finite accumulation points.
\end{enumerate}
One can describe these four cases geometrically. Specifically, we have 
\[
	(\tr \frakX_0)^{-1} \big( (-2, 2) \big) = 
	\bigcup_{j=1}^N I_j,
\]
where $I_j$ are disjoint open non-empty bounded intervals whose closures might touch each other. Let $I_j = (x_{2j-1}, x_{2j})$, where the sequence $(x_k : k = 1, 2, \ldots, 2N)$ is increasing. Then we are in the case~\ref{eq:PI} if $0$ belongs to some interval $I_j$, in the case~\ref{eq:PIIa} if $0$ lies on the boundary of exactly two intervals, in the case~\ref{eq:PIIb} if $0$ lies on the boundary of exactly one interval and in the case~\ref{eq:PIII} in the remaining cases. An example for $N=4$ is presented in Figure~\ref{img:perMod}.
\begin{figure}[h!]
	\centering
	\begin{tikzpicture}
		\draw[->, thin, black] (-5.5,0) -- (5.5, 0);  
		\draw[-,ultra thick, black!70] (-4.472135954,0) -- (-4,0);
		\draw[-,ultra thick, black!70] (-2,0) -- (2,0);
		\draw[-,ultra thick, black!70] (4,0) -- (4.472135954,0);
		\filldraw[black] (0,0) circle (1.6pt);
		\draw (0.05,-0.3) node {$x_4=x_5$};
		\filldraw[black] (-4.472135954,0) circle (1.6pt);
		\draw (-4.422135954,-0.3) node {$x_1$};
		\filldraw[black] (-4,0) circle (1.6pt);
		\draw (-3.95,-0.3) node {$x_2$};		
		\filldraw[black] (-2,0) circle (1.6pt);
		\draw (-1.95,-0.3) node {$x_3$};
		\filldraw[black] (2,0) circle (1.6pt);
		\draw (2.05,-0.3) node {$x_6$};
		\filldraw[black] (4,0) circle (1.6pt);
		\draw (4.05,-0.3) node {$x_7$};
		\filldraw[black] (4.472135954,0) circle (1.6pt);
		\draw (4.522135954,-0.3) node {$x_8$};
	\end{tikzpicture}
	\caption{An example for $N=4$. If $0 = x_4$, then we are in the case~\ref{eq:PIIa}, while $0 \in \{x_1, x_2, x_3, x_6, x_7, x_8 \}$ corresponds to the case~\ref{eq:PIIb}.}
	\label{img:perMod}
\end{figure}
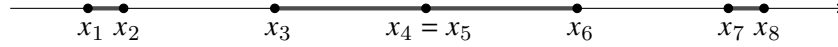

Below, we shall show that, under some regularity conditions, measures $\mu$ corresponding to $N$-periodically modulated Jacobi parameters satisfy Nevai's condition locally uniformly on $\supp(\mu) \setminus X$, where $X$ is some explicit finite set. This shows that the results from Section~\ref{sec:3} can be applied to such measures.

\subsubsection{Case \ref{eq:PI}} \label{sec:4.1.1}
Let $N \geq 1$. Suppose that $(a_n),(b_n)$ are $N$-periodically modulated Jacobi parameters such that 
$|\tr \frakX_0(0)| < 2$ and for some $r \geq 1$
\begin{equation} \label{eq:67}
	\bigg( \frac{a_{n-1}}{a_n} \bigg),
	\bigg( \frac{b_n}{a_n} \bigg),
	\bigg( \frac{1}{a_n} \bigg) \in \calD_r^N.
\end{equation}
Then according to \cite[Theorem A]{SwiderskiTrojan2019} the measure $\mu$ is determinate if and only if $\sum_{n=0}^\infty 1/a_n = \infty$. If this is the case, then by
\cite[Corollary 8]{SwiderskiTrojan2019} the measure $\mu$ is purely absolutely continuous on $\RR$ with a continuous positive density. 
Therefore, by Theorem~\ref{thm:3}, Nevai's condition holds uniformly on compact subsets of $\RR$.

Let us comment on some applications of our findings. First of all, by \cite[Theorem 1.2]{ChristoffelI},
\begin{equation} \label{eq:68}
	\lim_{n \to \infty} \frac{1}{\rho_n} K_n(x,x;\mu) = \frac{\upsilon(x)}{\mupAC(x)}
\end{equation}
locally uniformly with respect to $x \in \RR$, where(\footnote{For any $2 \times 2$ matrix $X$ we define its \emph{discriminant} by $\discr X = (\tr X)^2 - 4 \det X$.})
\begin{equation} \label{eq:69}
	\rho_n = \sum_{j=0}^{n-1} \frac{\alpha_j}{a_j} \quad \text{and} \quad
	\upsilon(x) \equiv \frac{1}{\pi N} \frac{|\tr \frakX_0'(0)|}{\sqrt{-\discr \frakX_0(0)}}.
\end{equation}
In view of \eqref{eq:68} the hypotheses of Corollary~\ref{thm:1} are satisfied for any compact $K \subset \RR$. Next, by \cite[Theorem 7.1]{zeros},
\begin{equation} \label{eq:70}
	\frac{1}{\rho_n} \frac{K_n(x,x;\mu)}{1+x^2} \ud \mu(x) \xrightarrow{w} \frac{1}{1+x^2} \upsilon(x) \ud \lambda(x).
\end{equation}
Therefore, the hypotheses of Theorem~\ref{thm:5} are satisfied for $w(x) = 1+x^2$. Observe that the hypotheses of Theorem~\ref{thm:4} are satisfied provided (cf. \eqref{eq:98})
\begin{equation} \label{eq:71}
    \lim_{n \to \infty} \frac{a_n}{n} = 0.
\end{equation}

Let us remark that Hermite, Freud, and Meixner--Pollaczek polynomials (see \cite[Section 5]{PeriodicII} for details) belong to the case~\ref{eq:PI} and satisfy \eqref{eq:67}.

\subsubsection{Case \ref{eq:PIIa}} \label{sec:4.1.3}

Let $N \geq 1$. Suppose that $(a_n),(b_n)$ are $N$-periodically modulated Jacobi parameters such that
$|\tr \frakX_0(0)| = 2$, and $\frakX_0(0)$ is diagonalizable. Then $\frakX_0(0) = \varepsilon \Id$ for some
$\varepsilon \in \{-1,1\}$. Suppose that
\begin{equation} \label{eq:73}
	\bigg( a_n \Big(\frac{a_{n-1}}{a_n} - \frac{\alpha_{n-1}}{\alpha_n} \Big) \bigg),
	\bigg( a_n \Big( \frac{b_n}{a_n} - \frac{\beta_n}{\alpha_n} \Big) \bigg),
	\bigg( \frac{1}{a_n} \bigg) \in \calD_1^N.
\end{equation}
Then according to \cite[Proposition 4]{ChristoffelII} the measure $\mu$ is determinate. Moreover,
\cite[Proposition 9]{PeriodicIII} implies that
\[
	h(x) = \lim_{j \to \infty} a_{jN+N-1}^2 \discr X_{jN}(x)
\]
exists. In view of the proof of \cite[Corollary 1]{PeriodicIII}, $h$ is a polynomial of degree $2$ with a negative leading coefficient. 
Let $\Lambda_- = h^{-1} \big( (-\infty,0) \big)$ and $\Lambda_+ = h^{-1} \big( (0, +\infty) \big)$. 
If 
\begin{equation} \label{eq:74}
	\lim_{n \to \infty} (a_{n+N} - a_n) = 0,
\end{equation}
then by \cite[Theorem D]{PeriodicIII} the measure $\mu$ is purely absolutely continuous on $\Lambda_-$ with a positive continuous density. Therefore, by Theorem~\ref{thm:3}, Nevai's condition holds uniformly on compact subsets of $\Lambda_-$. Next, by \cite[Theorem C]{Discrete}, the measure $\mu$ is purely discrete on $\Lambda_+$ and the only accumulation points of $\supp(\mu) \cap \Lambda_+$ are contained in the set $\partial \Lambda_- = \{x : h(x) = 0\}$. Since $h$ is a polynomial of degree $2$, the set $\partial \Lambda_-$ has at most two points. Finally, by Theorem~\ref{thm:2}, Nevai's condition holds uniformly on compact subsets of $\supp(\mu) \setminus \partial \Lambda_-$.

Let us comment on some applications of our findings. First of all, by \cite[Theorem 8 and Remark 1]{ChristoffelII}
\begin{equation} \label{eq:75}
	\lim_{n \to \infty} 
	\frac{1}{\rho_n} K_n(x,x;\mu) =
	\frac{\upsilon(x)}{\mupAC(x)}
\end{equation}
locally uniformly with respect to $x \in \Lambda_-$, where
\begin{equation} \label{eq:76}
	\rho_n = \sum_{j=0}^{n-1} \frac{\alpha_j}{a_j} \quad \text{and} \quad
	\upsilon(x) = \frac{1}{4 \pi N \alpha_{N-1}} \frac{|h'(x)|}{\sqrt{-h(x)}} \mathds{1}_{\Lambda_-}(x)
\end{equation}
(cf. \cite[formula (7.12)]{ChristoffelII}).
In view of \eqref{eq:75} the hypotheses of Corollary~\ref{thm:1} are satisfied for any compact $K \subset \Lambda_-$. Next, by \cite[Theorem 7.2]{zeros}
\begin{equation} \label{eq:77}
	\frac{1}{\rho_n} \frac{K_n(x,x;\mu)}{1+x^2} \ud \mu(x) \xrightarrow{w}
	\frac{1}{1+x^2} \upsilon(x) \ud \lambda(x).
\end{equation}
Therefore, the hypotheses of Theorem~\ref{thm:5} are satisfied for $w(x) = 1 + x^2$. To prove that the hypotheses of Theorem~\ref{thm:4} are satisfied, it is enough to prove that
\begin{equation} \label{eq:78}
	\lim_{n \to \infty} \frac{a_n}{n} = 0
\end{equation}
(cf. \eqref{eq:98}).
Observe that for any $i \in \{0,1,\ldots,N-1\}$ we have by Stolz-Ces\`{a}ro theorem and \eqref{eq:74}
\[
	\lim_{n \to \infty} \frac{a_{nN+i}}{nN+i} =
	\frac{1}{N} \lim_{n \to \infty} (a_{(n+1)N+i} - a_{nN+i}) = 0,
\]
which leads to \eqref{eq:78}.

Let us remark that generalized Hermite polynomials (see Example~\ref{ex:genHerm} for details) belong to the case~\ref{eq:PIIa} and they satisfy \eqref{eq:73} and \eqref{eq:74}.

\subsubsection{Case \ref{eq:PIIb}} \label{sec:4.1.4}

Let $N \geq 1$. Suppose that $(a_n),(b_n)$ are $N$-periodically modulated Jacobi parameters such that $\frakX_0(0)$ is non-diagonalizable. Let $\gamma=(\gamma_n : n \in \NN)$ be a sequence of positive numbers tending to infinity and satisfying 
\begin{equation} \label{eq:79}
	\bigg(\sqrt{\gamma_n} \Big(\sqrt{\frac{\alpha_{n-1}}{\alpha_n}} - \sqrt{\frac{\gamma_{n-1}}{\gamma_n}}\Big) \bigg),
	\bigg(\frac{1}{\sqrt{\gamma_n}} \bigg) \in \calD_1^N,
\end{equation}
and
\begin{equation} \label{eq:80}
	\lim_{n \to \infty} 
	\big(\sqrt{\gamma_{n+N}} - \sqrt{\gamma_n} \big) = 0.
\end{equation}
Suppose that
\begin{equation} \label{eq:81}
	\bigg(\sqrt{\gamma_n} \Big(\frac{\alpha_{n-1}}{\alpha_n} - \frac{a_{n-1}}{a_n} \Big) \bigg),
	\bigg(\sqrt{\gamma_n} \Big(\frac{\beta_n}{\alpha_n} - \frac{b_n}{a_n}\Big) \bigg),
	\bigg(\frac{\gamma_n}{a_n} \bigg) \in \calD_1^N
\end{equation}
and for $\varepsilon = \sign{\tr \frakX_0(0)}$ 
\begin{equation}  \label{eq:82}
	\bigg(
	\gamma_n\big(1- \varepsilon [\frakX_n(0)]_{11}\big)
	\Big(\frac{\alpha_{n-1}}{\alpha_n} - \frac{a_{n-1}}{a_n}\Big)
	-
	\gamma_n \varepsilon \Big(\frac{\beta_n}{\alpha_n} - \frac{b_n}{a_n}\Big) \bigg) \in \calD_1^N.
\end{equation}
Then according to \cite[Theorem A]{jordan2}
\[
	h(x) = \lim_{j \to \infty} \gamma_{jN+N-1} \discr X_{jN}(x)
\]
exists and is a polynomial of degree at most one. 
Let $\Lambda_- = h^{-1} \big( (-\infty,0) \big)$ and $\Lambda_+ = h^{-1} \big( (0, +\infty) \big)$. 
Suppose that $\Lambda_- \cup \Lambda_+ \neq \emptyset$ and that the measure $\mu$ is determinate (for detailed discussion of the second property, see \cite[Section 9]{jordan2}).

If $\Lambda_- \neq \emptyset$ we define
\begin{equation} \label{eq:83}
	\rho_n = \sum_{j=0}^{n-1} \frac{\sqrt{\alpha_j \gamma_j}}{a_j} \quad \text{and} \quad
	\upsilon(x) = \frac{\sqrt{\alpha_{N-1}}}{\pi N} \frac{|\tr \frakX_0'(0)|}{\sqrt{-h(x)}} \mathds{1}_{\Lambda_-}(x).
\end{equation}
Then according to \cite[Theorem B]{jordan2} the measure $\mu$ is determinate if and only if $\rho_n \to \infty$. Hence, according to \cite[Corollary 7.4]{zeros}, the measure $\mu$ is purely absolutely continuous on $\Lambda_-$ with a continuous positive density and 
\begin{equation} \label{eq:84}
	\lim_{n \to \infty} \frac{1}{\rho_n} K_n(x,x;\mu) = \frac{\upsilon(x)}{\mupAC(x)}
\end{equation}
locally uniformly with respect to $x \in \Lambda_-$. Consequently, by Theorem~\ref{thm:3}, Nevai's condition holds uniformly on compact subsets of $\Lambda_-$. Next, if $\Lambda_+ \neq \emptyset$, then by \cite[Theorem A]{jordan2} the measure $\mu$ is purely discrete on $\Lambda_+$ and the only accumulation points of $\supp(\mu) \cap \Lambda_+$ are contained in the set $\partial \Lambda_- = \{ x : h(x) = 0 \}$. Since $h$ is a non-zero polynomial of degree at most $1$, the set $\partial \Lambda_-$ has at most one point. Thus, by Theorem~\ref{thm:2}, Nevai's condition holds uniformly on compact subsets of $\supp(\mu) \setminus \partial \Lambda_-$. 

Let us comment on some applications of our findings. In view of \eqref{eq:84} the hypotheses of Corollary~\ref{thm:1} are satisfied for any compact $K \subset \Lambda_-$. Next, by \cite[Theorem 7.3]{zeros},
\begin{equation} \label{eq:85}
	\frac{1}{\rho_n} \frac{K_n(x,x;\mu)}{1+x^2} \ud \mu(x) \xrightarrow{w} \frac{1}{1+x^2} \upsilon(x) \ud \lambda(x).
\end{equation}
Therefore, the hypotheses of Theorem~\ref{thm:5} are satisfied for $w(x) = 1+x^2$. Moreover, the hypotheses of Theorem~\ref{thm:4} are satisfied  provided (cf. \eqref{eq:98})
\begin{equation} \label{eq:86}
	\lim_{n \to \infty} \frac{a_n}{n\sqrt{\gamma_n}} = 0.
\end{equation}

Let us remark that Laguerre-type polynomials (see Example~\ref{ex:genLag} for details) belong to the case~\ref{eq:PIIb} and they satisfy \eqref{eq:79}, \eqref{eq:80}, \eqref{eq:81}, \eqref{eq:82} for $\gamma_n = a_n$.

\subsubsection{Case \ref{eq:PIII}} \label{sec:4.1.2}
Let $N \geq 1$. Suppose that $(a_n),(b_n)$ are $N$-periodically modulated Jacobi parameters such that
$|\tr \frakX_0(0)| > 2$, and for some $r \geq 1$,
\begin{equation} \label{eq:72}
	\bigg( \frac{a_{n-1}}{a_n} \bigg),
	\bigg( \frac{b_n}{a_n} \bigg),
	\bigg( \frac{1}{a_n} \bigg) \in \calD_r^N.
\end{equation}
Then according to \cite[Theorem 5.3]{Discrete} (cf. \cite[Remark 5.4]{Discrete}) the measure $\mu$ is determinate, the measure $\mu$
is purely discrete on $\RR$ and $\supp(\mu)$ has no finite accumulation points. Thus, in view of Theorem~\ref{thm:2}, Nevai's condition holds uniformly on compact subsets of $\supp(\mu)$.

Since $\muAC = 0$, there seem to be no interesting applications of Section~\ref{sec:3} for this case.

Let us remark that Meixner polynomials (see Example~\ref{ex:Meixner} for details) belong to the case~\ref{eq:PIII} and they satisfy \eqref{eq:72}.

\subsection{Periodic blend} \label{sec:4.2}
We say that Jacobi parameters $(a_n),(b_n)$ are \emph{$N$-periodically blended}, if there exist $N$-periodic sequences $(\alpha_n),(\beta_n)$ of positive and real numbers, respectively, such that
\[
	\lim_{j \to \infty} |a_{j(N+2)+i} - \alpha_i| = 0, \quad
	\lim_{j \to \infty} |b_{j(N+2)+i} - \beta_i| = 0, \quad i=0,1,\ldots,N-1,
\]
and $b_{j(N+2)+N} = b_{j(N+2)+N+1} = 0$, and
\begin{equation} \label{eq:91}
	\lim_{j \to \infty} a_{j(N+2)+N} = \infty, \quad
	\lim_{j \to \infty} \frac{a_{j(N+2)+N}}{a_{j(N+2)+N+1}} = 1.
\end{equation}
Next, we define the $(N+2)$-step transfer matrix by
\[
	X_n(x) = B_{N+1+n}(x) \ldots B_{n+1}(x) B_n(x), \quad n \geq 1.
\]
Then by \cite[Section 3.3]{ChristoffelI} we have
\[
	\calX_1(x) = \lim_{j \to \infty} X_{j(N+2)+1}(x) =
	\begin{pmatrix}
		0 & -1 \\
		-\frac{\alpha_{N-1}}{\alpha_0} & -\frac{2x - \beta_0}{\alpha_0}
	\end{pmatrix}
	\frakB_{N-1}(x) \ldots \frakB_2(x) \frakB_1(x).
\]
Let $h(x) = \discr \calX_1(x)$ and set 
\[
	\Lambda_- = h^{-1} \big( (-\infty,0) \big) \quad \text{and} \quad
	\Lambda_+ = h^{-1} \big( (0,+\infty) \big).
\]
By \cite[Theorem 3.13]{ChristoffelI} 
\[
	\Lambda_- = \bigcup_{j=1}^N I_j,
\]
where $I_j$ are disjoint open non-empty bounded intervals whose closures might touch each other.

Notice that by \eqref{eq:91} the condition \eqref{eq:5} is violated. Therefore, measures $\mu$ corresponding to $N$-periodically blended Jacobi parameters have unbounded supports. This class has been introduced in \cite{Janas2011} as an explicit example of measures such that $\supp(\muAC)$ is the closure of $\Lambda_-$. Later, this class has been studied in more detail in \cite{Discrete, ChristoffelI, SwiderskiTrojan2019}.

Suppose that $(a_n),(b_n)$ are $N$-periodically blended Jacobi parameters such that for some $r \geq 1$
\[
	\bigg( \frac{1}{a_n} \bigg), 
	\bigg( \frac{b_n}{a_n} \bigg) 
	\in \calD_r^{N+2}
	\quad \text{and} \quad
	\bigg( \frac{a_{n(N+2)+N}}{a_{n(N+2)+N+1}} \bigg) 
	\in \calD_r^{1}.
\]
Then according to \cite[Corollary 9]{SwiderskiTrojan2019} the moment problem for $\mu$ is determinate and the measure $\mu$ is purely absolutely continuous on $\Lambda_-$ with a continuous positive density. Therefore, by Theorem~\ref{thm:3}, Nevai's condition holds uniformly on compact subsets of $\Lambda_-$. Next, by \cite[Theorem 5.1]{Discrete} (cf. \cite[Remark 5.2]{Discrete}), the measure $\mu$ is purely discrete on $\Lambda_+$ and the only accumulation points of $\supp(\mu) \cap \Lambda_+$ are contained in the set $\partial \Lambda_- = \{x : h(x) = 0\}$. Since $h$ is a polynomial of degree $2N$, the set $\partial \Lambda_-$ has at most $2N$ points. Thus, by Theorem~\ref{thm:2}, Nevai's condition holds uniformly on compact subsets of $\supp(\mu) \setminus \partial \Lambda_-$.

Let us recall that by \cite[Theorem 4.10]{ChristoffelI} (cf. Theorem~3.13, Corollary~3.15 and Remark~4.14 from \cite{ChristoffelI}) 
\begin{equation} \label{eq:90}
	\lim_{n \to \infty} \frac{1}{\rho_n} K_n(x,x;\mu) = \frac{\upsilon(x)}{\mupAC(x)}
\end{equation}
locally uniformly with respect to $x \in \Lambda_-$, where
\begin{equation} \label{eq:87}
	\rho_n = n \quad \text{and} \quad
	\upsilon(x) = \frac{1}{\pi (N+2)} 
	\frac{|\tr \calX_1'(0)|}{\sqrt{-\discr \calX_1(0)}} \mathds{1}_{\Lambda_-}(x).
\end{equation}
In view of \eqref{eq:90} the hypotheses of Corollary~\ref{thm:1} are satisfied for any compact $K \subset \Lambda_-$. Let us remark that weak convergence of the Christoffel--Darboux kernels for periodic blend has not been studied yet, but it seems that the techniques of \cite{zeros} can be applied here.

\subsection{Examples of explicit measures} \label{sec:4.3}
In this section we present a few examples of well-known measures with periodically modulated Jacobi parameters satisfying the appropriate regularity conditions from Section~\ref{sec:4.1}. 

\begin{example}[Freud weights] \label{ex:Freud}
For $\gamma \geq 1$, consider the probability measure
\[
	\ud \mu = c_{\gamma} \ue^{-|x|^\gamma} \ud \lambda(x),
\]
where $c_{\gamma}$ is the normalization constant. Notice that for $\gamma=2$ we recover a normal distribution. 
According to \cite[Example 3]{PeriodicII} the Jacobi parameters of the determinate measure $\mu$ satisfy \eqref{eq:65} with $N=1$ and $\alpha_n \equiv 1, \beta_n \equiv 0$. 
Then by \eqref{eq:66} we have
\begin{equation} \label{eq:88}
	\frakX_0(x) = 
	\begin{pmatrix}
		0 & 1 \\
		-1 & x
	\end{pmatrix}.
\end{equation}
Thus, $\tr \frakX_0(0) = 0$. Therefore, we are in the setup of Section~\ref{sec:4.1.1}. Moreover, condition~\eqref{eq:67} is satisfied for $r=1$. Let $\rho_n$ and $\upsilon$ be defined as in \eqref{eq:83}. Using \eqref{eq:88}, we immediately get $\upsilon(x) \equiv \frac{1}{2 \pi}$. By the Stolz--Ces\`{a}ro Theorem and \cite[Example 3]{PeriodicII} one can check that
\begin{equation} \label{eq:89}
	\lim_{n \to \infty} \frac{\rho_n}{\tilde{\rho}_n} = 1 \quad \text{where} \quad
	\tilde{\rho}_n = 
	\frac{1}{2} 
	\Bigg( 
		\frac
		{\Gamma \big( \tfrac{\gamma}{2} \big) \Gamma \big( \tfrac{1}{2} \big)} 
		{\Gamma \big( \tfrac{1+\gamma}{2} \big)}
	\Bigg)^{\tfrac{1}{\gamma}}
	\sum_{j=1}^n j^{-\tfrac{1}{\gamma}}.
\end{equation}
Moreover, by \cite[Example 3]{PeriodicII} the condition \eqref{eq:71} is satisfied if and only if $\gamma > 1$.
\end{example}

\begin{example}[Meixner polynomials] \label{ex:Meixner}
For any $s > 0$ and $p \in (0,1)$ the negative binomial distribution is the probability measure
\[
	\mu = (1-p)^s \sum_{k=0}^\infty \binom{k+s-1}{k} p^k \delta_k.
\]
Notice that for $s=1$ we recover geometric distributions.
The orthonormal polynomials in $L^2(\mu)$ are called \emph{Meixner polynomials} and they satisfy
\[
	a_n = \frac{\sqrt{(n+1)(n+s) p}}{1-p}, \quad b_n = \frac{n + (n+s)p}{1-p},
\]
see, e.g. \cite[formula (9.10.4)]{Koekoek2010}. It is easy to check that \eqref{eq:65} is satisfied for $N=1$ and $\alpha_n \equiv 1, \beta_n \equiv \sqrt{p} + \frac{1}{\sqrt{p}} > 2$. 
Then by \eqref{eq:66} we have
\[
	\frakX_0(x) = 
	\begin{pmatrix}
		0 & 1 \\
		-1 & x - \sqrt{p} - \frac{1}{\sqrt{p}}
	\end{pmatrix}.
\]
Thus, we are in the setup of Section~\ref{sec:4.1.2}. Moreover, it is easy to show that \eqref{eq:72} is satisfied for $r=1$.
\end{example}

\begin{example}[Generalized Hermite polynomials] \label{ex:genHerm}
For any $t > -1$ consider the probability measure
\[
	\ud \mu = \frac{1}{\Gamma(\tfrac{1+t}{2})} |x|^{t} \ue^{-x^2} \ud \lambda(x).
\]
Notice that for $t=0$ we recover a normal distribution. According to \cite[Example 1]{PeriodicII} the Jacobi parameters of $\mu$ satisfy \eqref{eq:65} with $N=1$ and $\alpha_n \equiv 1, \beta_n \equiv 0$. However, unless $t=0$ the condition~\eqref{eq:72} is \emph{not} satisfied for $N=1$. Again by \cite[Example 1]{PeriodicII}, the condition~\eqref{eq:73} is satisfied for $N=2$ and we are in the setup of Section~\ref{sec:4.1.3}. Let $\rho_n$ and $\upsilon$ be defined as in \eqref{eq:76}. By the Stolz--Ces\`{a}ro Theorem one can check that
\[
	\lim_{n \to \infty} \frac{\rho_n}{\tilde{\rho}_n} = 1 
	\quad \text{where} \quad
	\tilde{\rho}_n = \sqrt{2} \sum_{j=1}^n j^{-\tfrac{1}{2}}
\]
Moreover, \cite[Corollary 3]{PeriodicII} (see also \cite[Theorem 6]{ChristoffelII}) implies that $\upsilon(x) = \frac{1}{2 \pi} \mathds{1}_{\RR \setminus \{0\}}(x)$. Finally, condition~\eqref{eq:78} is satisfied.
\end{example}

\begin{example}[Laguerre-type polynomials] \label{ex:genLag}
For any $\gamma > -1$ and $\kappa \in \{2,3,\ldots \}$ consider the probability measure
\[
	\ud \mu = c_{\gamma,\kappa} x^\gamma \ue^{-x^\kappa} \mathds{1}_{(0,\infty)}(x) \ud \lambda(x),
\]
where $c_{\gamma,\kappa}$ is the normalizing constant. According to \cite[Section 10.1.1]{jordan} the Jacobi parameters of the measure $\mu$ satisfy \eqref{eq:65} with $N=1$ and $\alpha_n \equiv 1, \beta_n \equiv 2$. 
Then by \eqref{eq:66} we have
\[
	\frakX_0(x) = 
	\begin{pmatrix}
		0 & 1 \\
		-1 & x - 2
	\end{pmatrix}.
\]
Hence, we are in the setup of Section~\ref{sec:4.1.4}. Moreover, again by \cite[Section 10.1.1]{jordan}, the hypotheses \eqref{eq:79}, \eqref{eq:80}, \eqref{eq:81} and \eqref{eq:82} are satisfied for $\gamma_n = a_n$. Let $\rho_n$ and $\upsilon$ be defined as in \eqref{eq:83}. By the Stolz--Ces\`{a}ro Theorem and \cite[Section 10.1.1]{jordan} one can check that
\[
	\lim_{n \to \infty} \frac{\rho_n}{\tilde{\rho}_n} = 1 
	\quad \text{where} \quad
	\tilde{\rho}_n =
	\frac{1}{2} \bigg( \frac{2 (2\kappa)!!}{\kappa (2 \kappa-1)!!} \bigg)^{\tfrac{1}{2\kappa}}
	\sum_{j=1}^n j^{-\tfrac{1}{2 \kappa}}.
\]
Moreover, by \cite{jordan} (see Section 10.1.1, Corollary 3.4 and formula (3.1)) we have $h(x) = -4 x$, from which one can verify that $\upsilon(x) = \frac{1}{2 \pi} \frac{1}{\sqrt{x}} \mathds{1}_{(0, \infty)}(x)$. Finally, condition~\eqref{eq:86} is satisfied.
\end{example}

\begin{bibliography}{jacobi}
	\bibliographystyle{amsplain}
\end{bibliography}

\end{document}